\documentclass{amsart}

\usepackage{geometry}
\usepackage{amssymb}
\usepackage{url}
\usepackage{amsfonts}
\usepackage[dvipsnames]{xcolor}
\usepackage{amsthm}
\usepackage{graphicx}
\usepackage{float}

\theoremstyle{plain}

\newtheorem{theorem}{Theorem}[section]

\theoremstyle{definition}

\newtheorem{remark}[theorem]{Remark}

\newcommand{\Z}{{\mathbb Z}}
\newcommand{\R}{{\mathbb R}}

\begin{document}

\title[Orlicz-Lorentz Gauge Norm Inequalities for Nonegative Integral Operators]{Orlicz-Lorentz Gauge Norm Inequalities for Nonegative Integral Operators}

\author[R.~Kerman]{Ron Kerman}
\address{R.~Kerman\\Department of Mathematics and Statistics,
Brock University, 1812 Sir Isaac Brock Way, St. Catharines, ON L2S 3A1
}
\email{rkerman@brocku.ca}

\author[S.~Spektor]{Susanna Spektor}
\address{S.~Spektor\\Department of Mathematics and Statistics
Sciences, PSB,
Sheridan College Institute of Technology and Advanced Learning,
4180 Duke of York Blvd., Mississauga, ON L5B 0G5}
\email{susanna.spektor@sheridancollege.ca}

\maketitle

\begin{abstract}

\medskip

\noindent 2020 Classification: 46B06, 60C05
%

\noindent Keywords:kernel
\end{abstract}




\setcounter{page}{1}
\section{Introduction}

Let $k \in M_+(\R^n)$ and $K \in M_+(\R_+^{2n})$; here $\R_+=(0, \infty)$ and, for example, $M_+(\R_+^{2n})$ denotes the class of nonnegative, Lebesgue-measurable functions on the product space $\R_+^{2n}, n \in \Z_+$.

We consider two kinds of operators, namely, convolution operators $T_k$ defined at $f \in M_+(\R^n_+)$ by
\[
(T_kf)(x)=(k \ast f)(x)=\int_{\R^n}k(x-y)f(y)\, ds, \quad x \in \R^n,
\]
and general nonnegative integral operators $T_K$ defined at $f \in M_+(\R_+)$ by
\[
(T_Kf)(x)=\int_{\R_+}K(x,y)f(y)\, dy, \quad x \in \R_+.
\]

We are interested in Orlicz gauge norms $\rho_1$ and $\rho_2$ on $M_+(\R_+)$ for which
\begin{align}\label{1}
\rho_1((Tf)^*)\leq C \rho_2(f^*),
\end{align}
where $C>0$ is independent of $f$. In (\ref{1}), $T=T_k$ or $T=T_K$ and, accordingly, $f\in M_+(\R^n)$ or $f \in M_+(\R_+)$.

The function $f^*$ in (\ref{1}) is the nonincreasing rearrangement  of $f$ on $\R_+$, with
\[
f^*(t)=\mu_f^{-1}(t),
\]
in which
\[
\mu_f(\lambda)=|\{x:\, f(x)>\lambda\}|.
\]

The gauge norm $\rho$ is given in terms of an $N$-function
\[
\Phi(x)=\int_0^x\phi(y)\, dy, \quad x \in \R_+,
\]
$\phi$ a nondecreasing function mapping $\R_+$ onto itself, and a locally-integrable (weight) function $u$ on $\R_+$. Specifically, the gauge norm $\rho=\rho_{\Phi,u}$ is defined at $f\in M_+(\R_+)$ by
\[
\rho_{\Phi,u}(f)=\inf\{\lambda>0: \int_{\R_+}\Phi\left(\frac{f(x)}{\lambda}\right)u(x)\, dx\leq 1\}.
\]
Thus, in (\ref{1}), $\rho_1=\rho_{\Phi_1,u_1}$ and $\rho_2=\rho_{\Phi_2, u_2}$. The gauge norms in (\ref{1}) involving rearrangements are referred to as Orlicz-Lorentz norms.

\section{The convolution operators $T_k$ on $M_+(\R^n)$}

The fundamental inequality used in the analysis of $T_k$ is a rewritten form of the O'Neil rearranged convolution inequality.
\begin{align}\label{2.1}
\int_0^t(f \ast g)^*(s)\, ds\leq \int_0^tf^*(s)\, ds \int_0^tg^*(s)\, ds+t\int_t^{\infty}f^*(s)g^*(s)\, ds,
\end{align}
$f,g, \in M_+(\R^n_+), \, t \in \R_+$.

The inequality
\[
\rho_1((T_Kf)^*)\leq C \rho_2(f^*)
\]
is shown to following form
\begin{align}\label{3.1}
\rho_1(T_Lf^*)\leq C \rho_2(f^*), \quad f \in M_+(\R_+),
\end{align}
in which $L=L(x,y), \quad x,y, \in \R_+$ is the so-called iterated rearrangement of $K(x,y)$, which rearrangement is nonincreasing in each of $x$ and $y$.

Using the concept of the down dual of a gauge norm the inequalities (\ref{2.1}) and (\ref{3.1}) for nonincreasing fucntiosn are, in each case, reduced to gauge norm inequalities for general functions in Theorem 2.1 and Theorem 3.4, respectively. Sufficient conditions for stronger \textit{integral } inequalities like
\begin{align}\label{4.1}
\Phi_1^{-1}\left(\int_{\R_+}\Phi_1(\omega(x))((Tf)(x))t(x)\, dx\right)\leq \Phi_2^{-1}\left(\int_{\R_+}\Phi_2(u(y)f(y)v(y)\, dy)\right),
\end{align}
with $c>0$ independent of $f\in M_+(\R_+)$.

The integral inequality (\ref{4.1}) is the same as the corresponding gauge norm inequality
\[
\rho_{\Phi_1,t}(\omega Tf)\leq C \rho_{\Phi_2,v}(\omega f),
\]
when $\Phi_1(s)=s^q$ and $\Phi_2(s)=s^p, \quad 1<p\leq < \infty$. The necessary and sufficient conditions are spelled out in this case.

 Fundamental to our approach is the O'Neil's rearrangement convolution inequality which we write in  the form
 \begin{align}\label{2}
 \int_0^t(f \ast g)^*(s)\, ds \leq \int_0^t f^*(s)\, ds\int_0^*g^*(s)\, ds +t\int_t^{\infty}f^*(s)g^*(s)\, ds,
\end{align}
$f, g \in M_+(\R^n), \, t \in \R_+$.

We claim (\ref{2}) amounts to the domination of $(f\ast g)(t)$ by the expression
\begin{align}\label{3}
f^*(t)\int_0^tg^*+g^*(t)\int_0^tf^*+\int_t^{\infty}f^*g^*
\end{align}
in the HLP sense. Indeed,  the first two terms in (\ref{3}) add up to $\dfrac{d}{dt}\left(\int_0^tf^*\int_0^tg^*\right)$, so the integral between $0$ and $t$ of these terms is the first term on the right side of (\ref{2}). Again,
\begin{align*}
\int_0^t\int_s^{\infty}f^*g^*&=\int_0^t\int_s^t f^*g^*+t\int_t^{\infty}f^*g^*\\
&\leq \int_0^tf^*(s)\int_s^tg^*+t\int_t^{\infty}f^*g^*\\
&\leq \int_0^tf^*\int_0^tg^*+t\int_t^{\infty}f^*g^*.
\end{align*}

If $\rho=\rho_{\Phi}$ is the gauge norm, the HLP inequality
\[
\int_0^t(f \ast g)^*\leq 2\int_0^t\left[f^*(s)\int_0^s g^*+g^*(s)\int_0^sf^*+\int_s^{\infty}f^*g^*\right]\, ds
\]
ensures that
\[
\rho((f \ast g)^*)\leq 2\left[\rho_d\left(f^*(t)\int_0^tg^*\right)+\rho_d\left(g^*(t)\int_0^tf^*\right)+\rho_d\left(\int_t^{\infty}f^*g^*\right)\right],
\]
where
\[
\rho_d(h)=\sup_{\rho_{\Psi}(k)\leq 1}\int_{\R_+}hk^*, \quad h,k \in M_+(\R_+),
\]
with $\Psi(t)=\int_0^t \phi^{-1}(s)\, ds$ being the $N$-function complementary to $\Phi$.

According to \cite{GK} Theorem,
\[
\rho_d(h)=\rho\left(\int_0^th/t\right), \quad h \in M_+(\R_+).
\]

In Theorem 1 to follow, which summarizes the foregoing discussions, $(Ik^*)(t)=\int_0^tk^*$ and $(T_2k^*)(t)=I(Ik^*)(t)=\int_0^t(t-s)k^*(s)\, ds$.

\begin{theorem}\label{Th1}
Fix $k \in M_+(\R^n)$ and let $\Phi_1$ and $\Phi_2$ be $N$-functions, with $\Phi_2(2t)\approx \Phi_2(t), \quad t\gg 1$. Settings $\rho_i=\rho_{\Phi_i}, \quad i=1,2$, we have
\begin{align}\label{4}
\rho_1(T_kf)\leq C \rho_2(f), \quad f \in M_+(\R^n),
\end{align}
provided
\begin{align*}
(i)&\quad \rho_1\left(\frac 1t \int_0^tf(s)\int_s^tk^*\, ds\right)\leq C \rho_2(f), \quad f \in M_+(\R_+);\\
(ii)&\quad \tilde{\rho}_2\left(\frac 1t \int_0^tg(s)\int_s^tk^*\, ds\right)\leq C \tilde{\rho}_1(g), \quad g\in M_+(\R_+);\\
(iii)&\quad \tilde{\rho}_2\left(\frac{(I_2k^*)(t)}{t}\int_t^{\infty}g(y)\frac{dy}{y}\right)\leq C \tilde{\rho}_1(g), \quad g \in M_+(\R_+);\\
(iv)&\quad \hat{\rho}_2\left(\frac 1t \int_0^t\frac{(Ik^*)(s)}{s}g(s)\, ds\right)\leq C \hat{\rho_1}(g), \quad g \in M_+(\R),
\end{align*}
here $\hat{\rho}_i=\rho_{\Psi_i}, \quad i=1,2$.
\end{theorem}

There are no known conditions that are necessary and sufficient for any of the norm inequalities $(i)--(iv)$, unless $\Phi_1(t)=t^q, \, \Phi_2(t)=t^p, \quad 1<p\leq q< \infty$.

There are, however, such conditions for an \textit{integral} inequality stronger that the corresponding norm inequality. We combine Theorems 1.7 and 4.1 from \cite{BK} concerning such integral inequalities.

\begin{theorem}\label{TH2}
Consider $K(x,y)\in M_+(\R_+^2)$, which, for fixed $y \in \R_+$, increases in $x$ and, for fixed $x\in \R_+$, decreases in fixed $x$, and which satisfies the growth condition
\begin{align}\label{5}
K(x,y)\leq K(x,z)+K(z,y), \quad 0<y<z<x.
\end{align}

Let $t,u,v$ be nonnegative, measurable (weight) functions on $\R_+$ and suppose $\Phi_1$ and $\Phi_2$ are $N$-functions having complimentary functions $\Psi_1$ and $\Psi_2$, respectively, with $\Phi_1 \circ \Phi_2^{-1}$ convex. Then there exists $C>0$ such that
\begin{align}\label{6}
\Phi_1^{-1}\left(\int_{\R_+}\Phi_1(w(x)(T_Kf)(x))t(x)\, dx\right)\leq \Phi_2^{-1}\left(\int_{\R_+}\Phi_2(Cu(x)f(x))v(x)\, dx\right)
\end{align}
for all $f\in M_+(\R_+)$, if and only if there is a $c>0$, independent of $\lambda, x>0$, with
\begin{align}\label{7}
&\int_0^x\Psi_2\left(c\frac{\alpha(\lambda, x)K(x,y)}{\lambda u(y)v(y)}\right)v(y)\, dy\leq \alpha(\lambda, x)<\infty\nonumber\\
&\textit{and}\\
&\int_0^x\Psi_2\left(c\frac{\beta(\lambda, x)}{\lambda u(y)v(y)}\right)v(y)\, dy\leq \beta(\lambda, x)<\infty\nonumber,
\end{align}
where
\[
\alpha(\lambda, x)=\Phi_2 \circ \Phi_1^{-1}\left(\int_x^{\infty}\Phi_1(\lambda w(y))t(y)\, dy\right)
\]and
\[
\beta(\lambda,x)=\Phi_2 \circ \Phi_1^{-1}\left(\Phi_1(\lambda w(y)K(y,x))t(y)\, dy\right).
\]

In the case $K(x,y)=\chi_{(0,x)}(y)$ ony the first of the conditions in (\ref{6}) is required.
\end{theorem}
\begin{remark} The integral inequality (\ref{6}) with the kernel $K$ of Theorem \ref{TH2} replace by any $K \in M_+(\R_+^{2m})$ implies the norm inequality
\begin{align}\label{8}
\rho_{\Phi_1,t}(wT_Kf)\leq C \rho_{\Phi_2,v}(uf), \quad f \in M_+(\R_+^n).
\end{align}
Thus, in the generalization of (\ref{5}), replace $f$ by $\dfrac{f}{C \rho_{\Phi_2,v}(uf)}$ and suppose $\Phi_i(1)=1, \, i=1,2$. since $\int_{\R_+}\Phi_{2,v}\left(\dfrac{uf}{\rho_{\Phi_2,v}(uf)}\right)\leq 1,$ we get
\[
\int_{\R_+^n}\Phi_1\left(\frac{wT_Kf}{C \rho_{\Phi_{2,v}}(uf)}\right)t\leq 1,
\]
whence (\ref{8}) holds.
\end{remark}

In Theorem 4 to follow, $(Ik^*)(t)=\int_0^tk^*$ and $(I_2k^*)(t)=((I \circ I)k^*)(t)=\int_0^tk^*(s)(t-s)\, ds.$

\begin{theorem}\label{TH4}
Let $k, \rho_1$ and $\rho_2$ be as in Theorem \ref{Th1}. Then, (\ref{4}) holds if there exists $c>0$, independent of $\lambda, x \in \R_+$, such that
\begin{align*}
(v)\qquad\qquad\qquad \int_0^x \Psi_2\left(\frac{c \alpha(\lambda, x)}{\lambda}\int_y^xk^*\right)\, dy\leq \alpha_1(\lambda,x)<\infty
\end{align*}
and
\[
\Psi_2\left(\frac{c\beta_1(\lambda,x)}{\lambda}\right)\leq \beta_1(\lambda,x)<\infty,
\]
where
\[
\alpha_1(\lambda,x)=\Phi_2 \circ \Phi_1^{-1}\left(\int_x^{\infty}\Phi_1\left(\frac{\lambda}{y}\right)\, dy\right)
\]
and
\[
\beta_1(\lambda,x)\Phi_2 \circ \Phi_1^{-1}\left(\int_x^{\infty}\Phi_1\left(\frac{\lambda}{y} \int_x^yk^*\right)\, dy\right).
\]
\begin{align*}
(vi) \qquad\qquad\qquad \int_0^x\Phi_1\left(\frac{c \alpha_2(\lambda,x)}{\lambda}\int_y^xk^*\right)\, dy\leq \alpha_2(\lambda,x)< \infty
\end{align*}
and
\[
x \Phi_1\left(\frac{c \beta_2(\lambda,x)}{\lambda}\right)\leq \beta_2(\lambda,x)< \infty,
\]
where
\[
\alpha_2(\lambda,x)=\Psi_1 \circ \Psi_2^{-1}\left(\int_x^{\infty}\Psi_2\left(\frac{\lambda}{y}\right)\, dy\right)
\]
and
\[
\beta_2(\lambda,x)=\Psi_1 \circ \Psi_2^{-1}\left(\int_x^{\infty}\Psi_2 \left(\frac{\lambda}{y}\int_x^y k^*\right)\, dy\right)
\]
\begin{align*}
(vii)\qquad\qquad\qquad \int_0^x\Psi_2 \left(\frac{cy\alpha_3(\lambda,x)}{\lambda(I_2k^*)(y)}\right)\, dy\leq \alpha_3(\lambda,x)< \infty,
\end{align*}
where
\[
\alpha_3(\lambda, x)=\Phi_2 \circ \Phi_1^{-1}\left(\int_x^{\infty}\Phi_1 \left(\frac{\lambda}{y}\right)\, dy\right);
\]
\begin{align*}
(viii)\qquad\qquad\qquad \int_0^x \Phi_1 \left(\frac{cy\alpha_4(\lambda,x)}{\lambda(Ik^*)(y)}\right)\, dy\leq \alpha_4(\lambda, x)< \infty,
\end{align*}
where
\[
\alpha_4(\lambda,x)=\Psi_1 \circ \Psi_2^{-1}\left(\int_x^{\infty}\Psi_2\left(\frac{\lambda}{y}\right)\, dy\right).
\]
\end{theorem}
\begin{proof}
The conditions $(v)$ and $(vi)$ result from a direct application of Theorem \ref{TH2} to the operator with kernel
\[
K(x,y)=\int_y^{x}k^*,
\]
which clearly increases in $x$, decreases in $y$ and satisfies the growth condition (\ref{5}).

Writing the inequality $(iv)$ in the equivalent form
\[
\tilde{\rho}_2\left(\frac 1t (Ih)(t)\right)\leq C \tilde{\rho_1}(sh(s)/(Ik^*)(s)), \quad h \in M_+(\R_+),
\]
one sees the condition $(viii)$ comes out of Theorem 4.2 in \cite{BK}, which theorem accerts that for the Hardy operator $I$ only the first of the conditions in (\ref{7}) is required.

Finally, the inequality $(iii)$ is equivalent to the dual inequality
\[
\rho_1\left(\frac 1t \int_0^t (I_2k^*)(s)h(s)/s\, ds\right)\leq C \rho_2(h), \quad h \in M_+(\R_+),
\]
which can be written in the form
\[
\rho_1\left(\frac 1t (If)(t)\right)\leq C \rho_2(sf(s)/(I_2k^*)(s)), \quad f \in M_+(\R_+).
\]
The condition for this is $(vii)$.
\end{proof}

The integral inequality (\ref{6}) is equivalent to the norm inequality (\ref{8}) when $\Phi_1$ and $\Phi_2$ are power functions, say $\Phi_1(t)=t^q, \, \Phi_2(t)=t^p, \quad 1<p\leq q< \infty$.Moreover, in this case the conditions $(v)$ to $(viii)$ simplify. They become
\[
(v') \qquad\qquad\qquad \int_0^x\left(\int_y^x k^*\right)^{p'}\, dy\leq C x^{p'/q'}
\]
and
\[
\left(\int_x^{\infty}\left(\frac 1y \int_x^yk^*\right)^q\, dy\right)^{p'/q'}\leq C x^{-1};
\]
\[
(vi') \qquad\qquad\qquad \int_0^x\left(\int_y^x k^*\right)^q\, dy \leq C x^{q'/r}
\]
and
\[
\left(\int_x^{\infty}\left(\frac 1y \int_x^yk^*\right)^{p'\, dy}\right)^{q/r'}\leq C x^{-1};
\]
\[
(vii') \qquad\qquad\qquad \int-0^x\left(\frac{y}{(I_2k^*)(y)}\right)^{p'}\, dy\leq C x^{p/q'};
\]
\[
(viii') \qquad\qquad\qquad \int_0^x \left(\frac{y}{(Ik^*)(y)}\right)^q\, dy\leq Cx^{q/r}.
\]

It is shown in \cite{GK} that O'Neil condition inequality is sharp when $f$ and $g$ are radially decreasing on $\R^n$. Altogether then, we have
\begin{theorem}\label{TH5}
  Fix the indices $p$ and $q$, $1<p\leq q<\infty$, and suppose $k$ is radially decreasing on $\R^n$. Then,one has the inequality
  \[
  \left[\int_{\R^n}(T_kf)^q\right]^{1/q}\leq C \left[\int_\R f^p\right]^{1/p},
  \]
  with $C>0$ independent of $f \in M_+(\R^n)$, if and only if the conditions $(v') - (viii')$ hold.
\end{theorem}

We remark that a convolution operator $T_k$ whose kernel $k(x)=k(|x|)$decreases in $|x|$ on $\R^n$ is known as a potential operator; see \cite{KS} and the reference therein. In the formulas $(v') - (viii')$, $k^*(t)=k\left(\dfrac{\Gamma(n/2+1)^{1/n}}{\sqrt{\pi}}t^{1/n}\right), \quad t \in \R_+$.

\section{General nonnegative integral operators on $M_+(\R_+)$}

As a first step in our study of (\ref{1}) for $T=T_K$ we focus on the related inequality
\begin{align}\label{9}
\rho_1(T_Kf^*)\leq C \rho_2(f^*), \quad f \in M_+(\R_+).
\end{align}

\begin{theorem}\label{Th6}
Fix $K \in M_+(\R_+^2)$ and let $\Phi_1$ and $\Phi_2$ be $N$-function, with $\Phi_2(2t)\approx \Phi_2(t), \quad t \gg 1$. Given weight functions $u_1, u_2 \in M_+(\R_+), \, \int_{\R_+u_2=\infty}$, one has (\ref{9}) for $\rho_i=\rho_{\Phi_i, u_i}, \, i=1,2$, if
\begin{align}\label{10}
\rho_{\Psi_2, u_2}(Sg/u_2)\leq C \rho_{\Psi_1,u_1}(g/u_1), \quad g \in M_+(\R+),
\end{align}
\[
u_2(x)=\int_0^xu_2(z)\, dz, \quad \textit{and} \quad \Psi_i(t)=\int_0^t\phi_i^{-1}, \quad i=1,2.
\]
\end{theorem}
\begin{proof}
  The identity
  \[
  \int_{\R_+}gT_Kf^*=\int_{\R_+}f^*T'_{K}g
  \]
  readily yields that (\ref{9}) holds if and only if
  \[
  (\rho_1')^d(T_K'g)\leq C \rho_1'(g),
  \]
  where
  \begin{align}\label{11}
  \rho'_1(g)=\rho_{\Psi_1, u_1}(g/u_1)
  \end{align}
  and
  \begin{align}\label{12}
  (\rho'_2)^d(h)=\rho_{\Psi_2, u_2}\left(\int_0^xh/u_2(x)\right), \quad h \in M_+(\R_+).
  \end{align}
  For (\ref{12}), see [GK, Theorem 6.2]; (\ref{11}) is straightforward. The proof is complete on ranking $h=T'_K(g)$.
\end{proof}

To replace $T_Kf^*$ in (\ref{9}) by $(T_Kf)^*$ we will require
\[
\rho\left(t^{-1}\int_0^tf^*\right)\leq C\rho_1(f^*), \quad f \in M_+(\R_+).
\]
Conditions sufficient for such an inequality to hold are given in
\begin{theorem}\label{Th7}
Let $\Phi$ be an $N$-function satisfying $\Phi(2t)\approx \Phi(t), \, t\gg 1$, and suppose $u$ is weight on $\R_+$ with $\int_{\R_+}u=\infty$. Then,
\begin{align}\label{13}
\rho_{\Phi,u}\left(t^{-1}\int_0^tf^*\right)\leq C \rho_{\Phi,u}(f^*), \quad f \in M_+(\R_+),
\end{align}
provided
\[
\int_0^x\Phi(c \alpha(\lambda,x)/\lambda)u(y)\, dy\leq \alpha(\lambda,x)< \infty,
\]
with $c>0$ independent of $\lambda, x \in \R_+$, where
\begin{align}\label{14}
&\alpha(\lambda,x)=\int_0^{\infty}\Psi(\lambda u(y)/u(y))\, dy\nonumber\\
&\textit{and}\\
&\int_0^x\Psi(c\beta(\lambda,x)/\lambda U(y))u(y)\, dy\leq \beta(\lambda,x)< \infty,\nonumber
\end{align}
with $c>0$ independent of $\lambda,x>0$, where
\[
\beta(\lambda,x)=\int_x^{\infty}\Phi(\lambda/y)u(y)\, dy.
\]
\end{theorem}
\begin{proof}
In Theorem \ref{Th6}, take $K(x,y)=\chi_{0,x}(y)/x, \quad \Psi_1=\Psi_2=\Psi$ and $u_1=u_2=u$ to get
\[
(Sg)(x)=\int_0^x\int_y^{\infty}g(z)\frac{dz}{z}=\int_0^xg+x\int_x^{\infty}g(y)\frac{dy}{y},
\]
whence (\ref{10}) reduces to
\begin{align}\label{15}
&\rho_{\Psi,u}\left(\int_0^xg(U(x))\leq C \rho_{\Psi,u}(g/u)\right)\nonumber\\
&\textit{and}\\
&\rho_{\Psi,u}\left(x\int_0^{\infty}g(y)\frac{dy}{y}/U(x)\right)\leq C\rho_{\Psi,u}(g/u)\nonumber.
\end{align}
The first inequality in (\ref{15}) is a consequence of the modular inequality
\[
\int_{\R_+}\Psi\left(c\int_0^xg/u(x)\right)\leq \int_{\R}\Psi(g/u)u,
\]
which, according [BK, Theorem 4.1] hold, if and only if the first inequality in (\ref{14}) does.

Again, by duality, the second inequality in (\ref{15}) holds when the modular inequality
\[
\int_{\R_+}\Phi\left(c \frac 1x \int_0^xg\right)u\leq \int_{\R_+}\Phi(g/u)u
\]
does, which inequality holds if and only if one has the second condition in (\ref{14}).
\end{proof}

In Theorem \ref{Th8} below we show the boundedness of $T_Kf$ depends on that of $T_Kf^*$, where the kernel $L$ is the iterated rearrangement of $K$ considered in \cite{B}. Thus, for each $x \in \R_+$, we rearrange the function $k_x(y)=K(x,y)$ with respect to $y$ to get $(k^*_x)(s)=K^{*_2}(x,s)=k_s(x)$ and then rearrange the function of $x$ so obtained to arrive at $(K^{*_2})^{*_1}(t,s)=L(t,s)$. It is clear from its construction that $K(t,s)$ is nonincreasing in each of $s$ and $t$.

\begin{theorem}\label{Th8}
Consider $K \in M_+(\R_+^2)$ and set $L(t,s)=(K^{*_2})^{*_1}(t,s), \quad s,t \in \R_+$. suppose $\Phi_1$ and $\Phi_2$ are $N$-functions, with $\Phi_1(2t)\approx \Phi_1(t),\quad t\gg 1$, and let $u_1$ and $u_2$ be weight functions, with $\int_{\R_+}u_1=\infty$. Then, given the conditions (\ref{12}) for $\Psi=\Psi_1$ and $u=u_1$ one has
\[
\rho_1((T_kf)^*)\leq C \rho_2(f^*),
\]
provided
\[
\rho_1((T_Lf)^*)\leq C \rho_2(f^*), \quad f \in M_+(\R_+).
\]
\end{theorem}
\begin{proof}
We claim
\begin{align}\label{16}
(T_Kf)^{**}(t)\leq (T_Lf^*)^{**}(t), \quad t \in \R_+,
\end{align}
in which, say, $(T_Kf)^{**}(t)=t^{-1}\int_0^t(T_Kf)^*$.

Indeed, given $E \subset \R_+, \quad |E|=t$,
\begin{align*}
\int_ET_Kf&\leq \int_E\int_{\R_+}K^{*_2}(x,s)f^*(s)\,ds\\
&=\int_{\R_+}f^*(s)\, ds\int_{\R_+}\chi_E(x)K^{*_2}(x,s)\, dx\\
&\leq \int_{\R_+}f^*(s)\, ds \int_0^tL(u,s)\, du\\
&=\int_0^t(T_Lf^*)(u)\,du.
\end{align*}
Taking the supremum over all such $E,$ then dividing by $t$ wields (\ref{16}).

Next, the inequality
\[
\rho_{\Phi_1,u_1}((T_Kf)^{**})\leq C \rho_{\Phi_1, u_1}((T_Lf^*)^{**})
\]
is equivalent to
\[
\rho_{\Phi_1, u_1}((T_Kf)^*)\leq C \rho_{\Phi_1, u_1}(T_Lf^*),
\]
given (\ref{14}) for $\Phi=\Phi_1$ and $u=u_1$. For, in that case,
\begin{align*}
\rho_{\Phi_1,u_1}((T_Kf)^*)&\leq \rho_{\Phi_1,u_1}((T_kf)^{**})\\
&\leq \rho_{\Phi_1,u_1}((T_Kf^*)^{**})\\
&\leq C \rho_{\Phi_1,u}(T_Lf^*).
\end{align*}
The assertion of the theorem now follows.
\end{proof}

\begin{theorem}\label{Th9}
Let $K, L \Phi_1, \Phi_2, u_1$ and $u_2$ be as in theorem \ref{Th8}. Assume, in addition, that $\Phi_2(2t)\approx\Phi_2(t), \quad t\gg 1,$ $\int_{\R_+}u_2=\infty$ and that conditions (\ref{15}) hold for $\Phi=\Phi_1, \, u=u_1$. then,
\[
\rho_{\Phi_1,u_1}((T_Kf)^*)\leq C \rho_{\Phi_2, u_2}(f^*)
\]
provided
\begin{align}\label{17}
&\rho_{\Psi_2, u_2}(H_1f/u_2)\leq C \rho_{\Psi_1, u_1}(f/u_1)\nonumber\\
&\textit{and}\\
&\rho_{\Phi_1, i^2, u_1\circ i}(H_2g)\leq C \rho_{\Phi_2,i^2u_2\circ i}(g(i^2)), \quad f , \in M_+(\R_+)\nonumber,
\end{align}
where
\[
(H_1f)(x)=\int_0^xM_1(x,y)f(y)\, dy
\]
and
\[
(H_2g)(y)=\int_0^yM_2(y,x)g(x)\, dx,
\]
with $i(x)=x^{-1}$,
\[
M_1(x,y)=\int_0^xL(y,z)\, dz \quad \textit{and} \quad M_2(y,x)=\int_0^{x^{-1}}L(y^{-1},z)\, dz.
\]
\end{theorem}
\begin{proof}
In view of theorem \ref{Th8} , we need only verify conditions (\ref{17}) imply
\begin{align}\label{18}
\rho_{\Psi_1,u_1}(T_Lf^*)\leq C \rho_{\Psi_2, u_2}(f^*), \quad f \in M_+(\R_+).
\end{align}
Now, according to Theorem \ref{Th6}, (\ref{18}) will hold if one has (\ref{10}) with $K=L$. Again, (\ref{10}) is equivalent to the dual inequality
\[
\rho_{\Phi_1, u_1}(S'g)\leq C \rho_{\Phi_2, u_2}(gU_2/u_2),
\]
where
\begin{align*}
(S'g)(y)&=\int_0^{\infty}\left[\int_0^xL(y,z)\, dz\right]g(x)\, dx\\
&=\left(\int_0^y+\int_y^{\infty}\right)\left[\int_0^x L(y,z)\, dz\right]g(x)\, dx\\
&=\int_y^{\infty}\left[\int_0^xL(y,z)\, dz\right]g(x)\, dx
\end{align*}
has associate operator
\[
\int_0^x\left[\int_0^xL(y,z)\, dz\right]f(y)\, dy=(H_1f)(x)
\]
and this operator is to satisfy
\[
\rho_{\Psi_2, u_2}(H_1, U_2)\leq C \rho_{\Psi_1, u_1}(f/u_1).
\]

Again, if there is to exist $C>0$ so that
\begin{align}\label{19}
\rho_{\Phi_1, u_1}\left(\int_0^y\left[\int_0^x L(y,z)\, dz\right]g(x)\right)\leq C \rho_{\Phi_2, u_2}(gU_2/u_2),
\end{align}
then, for such $C$,
\[
\int_0^{\infty}\Phi_1\left(\int_0^y\left[\int_0^xL(y,z)\, dz\right]g(x\, dx)/C\rho_{\Phi_2, u_2}(gU_2/u_2)\right)u_1(y)\, dy\leq 1
\]
for all $g \in M_+(\R_+), \, g \neq 0 a.e.$ That is, on making the changes of variables $y \to y^{-1}$ then $x \to x^{-1}$, one will have
\[
\int_0^{\infty}\Phi_1\left(\int_0^y\left[\int_0^{x^{-1}}L(y^{-1},z)\, dz\right]\frac{g(x-1)}{x^2}\frac{dx}{x}/C\rho_{\Phi_1, u_2}(gU_2/u_2)\right)u_1(y^{-1})\frac{dy}{y}\leq 1.
\]
\[
\int_0^{\infty}\Phi_2(gU_2/\lambda u_2)u_2=\int_0^{\infty}\Phi_2(g \circ i u_2 \circ i/\lambda u_2 \circ i)i^2u_2\circ i,
\]
so
\[
\rho_{\Phi_2, u_2}(gU_2/u_2)=\rho_{\Phi_2, i^2u_2\circ i}(g \circ i u_2\circ i/u_2 \circ i),
\]
whence (\ref{19}) amounts to
\[
\rho_{\Phi_1, i^2u_1\circ i}(H_2(i^2g\circ i u_2 \circ i/u_2\circ i))\leq C \rho_{\Phi_2, i^2u_2}(g\circ u_2 \circ i/u_2 \circ i)
\]
or
\[
\rho_{\Phi_1, i^2u_1\circ i}(H_2g)\leq C \rho_{\Phi_2, i^2u_2\circ i}(g/i^2),
\]
since $i^2g\circ i u_2 \circ i/u_2 \circ i$ is arbitrary.
\end{proof}

We have to this point shown that the inequality (\ref{1}) holds for $\rho_i=\rho_{\Phi_i, u_i, \quad i=1,2}$, whenever the inequalities (\ref{17}) holds for $H_1$ and $H_2$. In Theorem \ref{Th10} below we give four conditions which, together with (\ref{14}) for $\Phi_1$ and $\Phi_2$, guarantee (\ref{17}).

The kernel $M_1(x,y)$ of the operator $H_1$ is increasing in $x$ and decreasing in $y$. Similarly, the kernel $M_2(y,x)$ of $H_2$ is increasing in $y$ and decreasing in $x$. The operators $H_1$ and $H_2$ will be so-called generalized Hardy operator (GHOs) if their kernels satisfy the growth conditions
\begin{align}\label{20}
&M_1(x,y)\leq M_1(x,z)+M_1(z,y), \quad y<z<x,\nonumber\\
&\textit{and}\\
&M_2(y,x)\leq M_2(y,z)+M_2(z,x), \quad x<z<y.\nonumber
\end{align}
Neither of the conditions in (\ref{20}) are guarantees to hold. They have to be assumed in theorem \ref{Th10} below so that we may apply Theorem 1.7 in \cite{BK} concerning GHOs. Theorem \ref{Th11} in the next section gives a class of kernels for which (\ref{20}) is satisfied.
\begin{theorem}\label{Th10}
Let $K, L, \Phi_1, \Phi_2, u_1, u_2, M_1, M_2, H_1$ and $H_2$ be as in Theorem \ref{Th9}. Assume, in addition, that $\Phi_1\circ \Phi_2^{-1}$ is convex and that $M_1$ and $M_2$ satisfy the growth conditions (\ref{20}). Then, one has
\begin{align}\label{21}
\int_{\R_+}\Phi_1(T_Kf)^*u_1\leq C \int_{\R_+}\Phi_2(f^*)u_2,
\end{align}
provided
\begin{align}\label{22}
&\int_0^x\Phi_1\left(c\frac{\alpha_1(\lambda, x)M_1(x,y)}{\lambda u_1(y)}\right)u_1(y)\, dy\leq \alpha_1(\lambda,x)<\infty,\nonumber\\
&\int_0^x\Phi_1\left(c\frac{\beta_1(\lambda,x)}{\lambda u_1(y)}\right)u_1(y)\, dy\leq \beta_1(\lambda,x)<\infty,\nonumber\\
&\int_0^y\Psi_2\left(c \frac{\alpha_2(\lambda,y)M_2(y,x)}{u_2(x^{-1})}\right)x^{-2}u_2(x^{-1})\, dx\leq \alpha_2(\lambda,y)<\infty\\
&\textit{and}\nonumber\\
&\int_0^4\Psi_2\left(c \frac{\beta_2(\lambda,y)}{u_2(x^{-1})}\right)x^{-2}u_2(x^{-1})\, dx\leq \beta_2(\lambda, y).
\end{align}
Here,
\begin{align*}
&\alpha_1(\lambda,x)=\Psi_1 \circ \Psi_2^{-1}\left(\int_x^{\infty}\Psi_2\left(\frac{\lambda}{u_2(y)}\right)u_2(y)\, dy\right),\\
&\beta_1(\lambda,x)=\Psi_1\circ\Psi_2^{-1}\left(\int_x^{\infty}\Psi_2\left(\frac{\lambda M_1(x,y)}{u_2(y)}\right)u_2(y)\, dy\right)\\
&\alpha_2(x,y)=\Phi_2\circ\Phi_1^{-1}\left(\int_y^{\infty}\Phi_1(\lambda)x^{-2}u_1(x^{-1})\, dx\right)\\
&\textit{and}\\
&\beta_2(\lambda, y)=\Phi_2\circ \Phi_1^{-1}\left(\int_y^{\infty}\Phi_1(\lambda M_2(x,y))x^{-2}u_1(x^{-1})\, dx\right).
\end{align*}
\end{theorem}
\begin{proof}
We prove the result involving $H_2$; the proof for $H_1$ is similar. Now, the norm inequality for $H_2$ for holds if one has the integral inequality
\begin{align}\label{23}
\Phi_1^{-1}\left(\int_{\R_+}\Phi_1\left(\frac{H_2g}{C}\right)y^{-2}u_1(y^{-1})\,dy\right)\leq \Phi_2^{-1}\left(\int_{\R_+}\Phi_2(y^2g(y))y^{-2}u_2(y^{-1})\, dy\right).
\end{align}
Indeed, in the latter replace $g(y)$ y $g(y)/\rho_{\Phi_2,y^{-2}u_2(y^{-1})}=g(y)/\lambda$, to get
\begin{align*}
&\int_{\R_+}\Phi_1\left(\frac{H_2g}{C \lambda}\right)y^{-2}u_1(y^{-1})\, dy \leq \Phi_1 \circ \Phi_2^{-1}\left(\int_{\R_+}\Phi_2\left(\frac{y^2g(y)}{\lambda}\right)y^{-2}u_2(y^{-1})\, dy\right)=\Phi_1\circ \Phi_2^{-1}(1)=1,
\end{align*}
where we have assumed, without loss of generality, that $\Phi_1(1)=\Phi_2(1)=1$. Hence,
\[
\rho_{\Phi_1, y^{-2}u_1(y^{-1})}(H_2g)\leq C \lambda=C \rho_{\Phi_2, y^{-2}u_2(y^{-1})}(g/i^2).
\]
But, (\ref{23}) is valid if and only if the third and fourth conditions in (\ref{23}) hold.
\end{proof}

\section{Examples}
\begin{theorem}\label{Th11}
Let $k$ be nonnegative, nonincreasing function on $\R_+$. Then, the growth conditions (\ref{20}) are satisfied for $K(x,y)=k(x+y), \, x,y \in \R_+$.
\end{theorem}
\begin{proof}
We observe that $K(x,y)=L(x,y)$ since $K$ decreases in each of $x$ and $y$. So,
\[
M_1(x,y)=\int_0^xk(y+s)\, ds \qquad \textit{and} \qquad M_2(x,y)=\int_0^{x-1}k(y^{-1}+s)\, ds.
\]

Now, given $y<z<x$,
\begin{align*}
M_1(x,y)&=\int_0^xk(y+s)\, ds=\int_0^zk(y+s)\, ds+\int_z^xk(y+s)\, ds\\
&=M_1(y,z)+\int_0^{x-z}k(y+z+s)\, ds\\
&\leq M_1(y,z)+\int_0^xk(z+s)\, ds\\
&=M_1(y,z)+M_1(x,z).
\end{align*}

Again, given $x<z<y,$
\begin{align*}
M_2(x,y)&=\int_0^{x^{-1}}k(y^{-1}+s)\, ds=\int_0^{z^{-1}}k(y^{-1}+s)\, ds+\int_{z^{-1}}^{x^{-1}}k(y^{-1+s})\, ds\\
&=M_2(z,y)+\int_0^{x^{-1}-z^{-1}}k(y^{-1}+z^{-1}+s)\, ds\\
&\leq M_2(z,y)+\int_0^{x^{-1}}k(z^{-1}+s)\, ds\\
&=M_2(z,y)+M_2(x,z).
\end{align*}
\end{proof}

\begin{theorem}\label{Th12}
Fix the indices $p$ and $q, \, 1<p\leq q< \infty$, and suppose $K(x,y)=k(x+y), \, x, y \in \R_+,$ where $k$ is nonnegative and nonincreasing on $\R_+$.
\begin{align}\label{24}
\left(\int_{\R_+}(T_Kf)^q\right)^{1/q}\leq C \left(\int_{\R_+}f^p\right)^{1/p},
\end{align}
with $C>0$ independent of $f \in M_+(\R_+)$, if and only if
\begin{align}\label{25}
&c\int_0^xM_1(x,y)^qu_1(y)^{1-q}\, dy\leq \lambda_1(x)^{1-q}\nonumber\\
&c\int_0^xu_1(y)^{1-q}\, dy\leq \beta_1(\lambda)^{1-q}\nonumber\\
&c\int_0^yM_2(y,x)^{p'}x^{-2}u_2(\lambda^{-1})^{1-p'}\leq \alpha_2(x)^{1-p'}\\
&\textit{and}\nonumber\\
&c\int_0^yx^{-2}u_2(x^{-1})^{1-p'}\leq \beta_2(x)^{1-p'}.
\end{align}
Here,
\begin{align*}
&\alpha_1(x)=\left(\int_x^{\infty}U_2(y)^{-p'}u_2(y)\, dy\right)^{q'/p'},\\
&\beta_1(x)=\left(\int_x^{\infty}\left[\frac{M_1(x,y)}{U_2(y)}\right]^{p'}u_2(y)\, dy\right)^{q'/p'},\\
&\alpha_2(y)=\left(\int_y^{\infty}x^{-2}u_1(x^{-1})\, dx\right)^{p/q}\\
&\textit{and}\\
&\beta_2(y)=\left(\int_y^{\infty}M_2(x,y)^qx^{-2}u_1(x^{-1})\, dx\right)^{p/q}.
\end{align*}
\end{theorem}
\begin{proof}
The $N$-functions $\Phi_1(t)=t^q$ and $\Phi_2(t)=t^p,$ as well as the weights $u_1$ and $u_2$ satisfy the conditions required in theorem \ref{Th10}. According to Theorem \ref{Th11}, so do the kernels $M_1$ and $M_2$ of $H_1$ and $H_2$, respectively. We conclude, then, that (\ref{21}) holds for $T_K$ given (\ref{22}), which in our case are the inequality (\ref{24}) and the conditions (\ref{25}). We observe that $\lambda$ cancels out in the latter conditions and we are left with $\alpha_1)1,x(=\alpha_1(x),$ etc.
\end{proof}

Consider a kernel of the form $K(x,y)=k(\sqrt{x^2+y^2})$, where $k(t)$ is nonincreasing in $t$ on $\R_+$ and $\int_0^{\infty}K(x,y)\, dy$ for all $a,x>0$. In particular, $K(x,y)$ is nonincreasing on $\R_+$ in each of $x$ and $y$. Again, $M_1(x,y)=\int_0^xk(\sqrt{x^2+z^2})\, dz$ and $M_2(x,y)=\int_0^yk(\sqrt{x^2+z^2})\, dz$ satisfy (\ref{20}), so that (\ref{24}) holds for $T_K$, given (\ref{25}).

In particular, the above is true for $K(x,y)=\left(x^2+y^2\right)^{-3/4}$. However, this kernel does not satisfy the classic Kantorovi\v{c} condition usually involved to prove (\ref{24}) for $T_K$. Indeed,
\[
 2^{-3/4}{(x+y)^{-3/2}}\leq K(x,y)\leq  2^{3/4} (x+y)^{-3/2},
\]
whence, for $p>1$,
\[
\left[\int_0^{\infty}K(x,y)^{p'}\, dy\right]^{1/p'}\approx x^{-1/p-1/2},
\]
and, therefore,
\[
\left[\int_0^{\infty}\left[\int_0^{\infty}K(x,y)^{p'}\, dy\right]^{q/p'}\, dx\right]^{1/q}=\left[\int_0^{\infty}\frac{dx}{x^{q/p+q/2}}\right]^{1/q}=\infty.
\]
The kernel $K(x,y)$ is not homogeneous of degree $-1$, so Theorem in \cite{HLP} does not apply to it.

We observe that in Theorem 3.3, the weights $v_i\equiv 1, \quad i=1,2$, then the inequality $\rho_1((T_Kp)^q)\leq C \rho_2(f^*)$ is the same as $\rho_1(T_Kf)\leq C \rho_2(f)$.

Finally, R.~O'Neil in \cite{O} proved that, for $K \in M_+(\R_+^2)$, one has, for each $f \in M_+(\R_+)$,
\[
\frac 1x \int_0^x(T_Kf^*)(y)\, dy\leq \int_0^{\infty}K^*(xy)f^*(y)\, dy.
\]
Given $K(x,y)=k(\sqrt{x^2+y^2})$ as above, $K^*(t)=k(t^{1/2})$, so the right side of the O'Neil inequality is
\[
\int_0^{\infty}k(\sqrt{xy})f^*(y)\, dy.
\]
On the other hand,
\[
(T_Kf^*)(x)=\int_0^{\infty}k(\sqrt{x^2+y^2})f^*(y)\, dy.
\]
Thus, if $\Phi_1(2t)\approx\Phi_1(t)$, for $t\gg 1$, the O'Neil inequality yields
\[
\int_0^{\infty}\Phi_1((T_Kf^*)(x))u_1(x)\, dx\leq \int_0^{\infty}\Phi_1\left(\int_0^{\infty}k(\sqrt{xy}f^*(y))\, dy\right)\, dx.
\]
Observing that $k(\sqrt{x^2+y^2})=k\left(\sqrt{\frac{x^2+y^2}{xy}}\sqrt{xy}\right)=k\left(\sqrt{\frac yx+ \frac xy}\sqrt{xy}\right)$, thus our bound is tighter.


\end{document}